\documentclass[12pt]{article}
\usepackage{amsmath,amsthm,amsfonts,amssymb}
\usepackage{latexsym,epsfig,subfigure,afterpage,graphics,epsfig,color}
\usepackage{amsbsy, amsopn,amstext,amsxtra,euscript,amscd}

\newtheorem{theorem}{Theorem}

\newtheorem{cor}[theorem]{Corollary}

\newtheorem{thm}{Theorem}
\newtheorem{lem}[thm]{Lemma}

\def\cC{{\mathcal C}}

\def\cH{{\mathcal H}}

\def\cP{{\mathcal P}}
\def\cQ{{\mathcal Q}}

\def\fA{{\mathfrak A}}
\def\fU{{\mathfrak U}}
\def\fV{{\mathfrak V}}

\def\F{{\mathbb F}}

\def\R{{\mathbb R}}

\def\Z{{\mathbb Z}}

\def\\{\cr}
\def\({\left(}
\def\){\right)}
\def\[{\left[}
\def\]{\right]}
\def\<{\langle}
\def\>{\rangle}
\def\fl#1{\left\lfloor#1\right\rfloor}

\def\mand{\qquad\mbox{and}\qquad}

\def\vec#1{\mathbf{#1}}

\def\Hm{{\cH}_{a}(m)}
\def\GL{\mathrm{GL}}

\def\NU{N(a,m;U,V)}

\def\Cm{{\cC}_{s}(a,m)}

\def\vol{\mathrm{vol}\,}

\begin{document}

\title{On the Convex Hull of the Points on Modular Hyperbolas}

\author{\sc{Sergei V.~Konyagin}\\
Steklov Mathematical Institute\\
8, Gubkin Street, Moscow, 119991, Russia\\
{\tt konyagin@mi.ras.ru}\\
\and
{\sc Igor E.~Shparlinski} \\
{Department of Computing, Macquarie University} \\
{Sydney, NSW 2109, Australia} \\
{\tt igor.shparlinski@mq.edu.au}}

\date{\today}
\maketitle

\begin{abstract} Given  integers $a$ and $m\ge 2$, let $\Hm$ be the following set
of integral points
$$
\Hm= \{(x,y) \ : \  xy \equiv a \pmod m,\ 1\le x,y \le m-1\}
$$

We improve several previously known  upper bounds on $v_a(m)$, the
number of vertices of the convex closure of $\Hm$, and show that
uniformly over all $a$ with $\gcd(a,m)=1$ we have $v_a(m) \le
m^{1/2 + o(1)}$ and furthermore, we have $v_a(m) \le
m^{5/12 + o(1)}$ for $m$ which are almost squarefree.
\end{abstract}



\section{Introduction}

For integers $a$ and $m\ge 2$, we define the modular hyperbola, $\Hm$, to be the set
of integral points
$$
\Hm= \{(x,y) \ : \  xy \equiv a \pmod m,\ 1\le x,y \le m-1\}.
$$

A systematic study of geometric properties of the set $\Hm$ has
been initiated in~\cite{FKSY} and continued in a number of works,
see~\cite{ChanShp,CillGar,FKS,KSY,Shp1,Shp3,ShpWin1} and
references therein, where also several surprising links to various
number theoretic questions have been discovered.

In particular, following~\cite{FKS,KSY}, we consider the convex closure $\Cm$
of the set $\Hm$ and let $v_a(n)$ denote the number of vertices of $\Cm$.

For $a=1$, it is shown in~\cite{KSY} that
\begin{equation}
\label{eq:v 3/4}
v_1(m) \le m^{3/4 + o(1)},
\end{equation}
which has been improved in~\cite{FKS} as
\begin{equation}
\label{eq:v 7/12}
v_1(m) \le m^{7/12 + o(1)},
\end{equation}
by using the bound $O(S^{1/3})$ of G.~Andrews~\cite{Andrews} on the number of  vertices
of a convex  polygon of area $S$ vertices on the integral lattice $\Z^2$.
In~\cite{KSY} a number of
other lower and upper bounds on  $v_1(m)$ have been established, which
however apply only to special classes of integers $m$. For example,
it shown in~\cite[Theorem~3.2]{KSY}  that for all $m > 1$,
$$
v_1(m) \ge  2 (\tau(m-1)-1)
$$
where $\tau(k)$ is the number of positive integer divisors of $k$, and
this estimate is tight as  $$
\#\{m \le x \ : \ v_1(m) = 2 (\tau(m-1)-1)  \} \gg \frac{x}{\log x},
$$
where, as usual, the notations  $U \ll V$ and  $V \gg U$ are
equivalent to $U = O(V)$  (throughout the paper, except
Lemma~\ref{lem:Curve}, the implied constants are absolute).
  Besides, one
can find in~\cite{KSY} an extensive numerical study of  $v_1(m)$  which
shows a somewhat mysterious behaviour  which exhibits both some chaotic and
regular aspects.

It has also been noticed in~\cite{FKS} that~\cite[Theorem~1]{Shp3}
implies that
\begin{equation}
\label{eq:v 1/2 a.a.}
v_a(m) \le m^{1/2 + o(1)},
\end{equation}
for all but $o(\varphi(m))$ integers $a$ with $1\le a  \le m-1$ and
$\gcd(a,m)=1$, where, as usual, $\varphi(m)$ denotes the Euler function.

Here we use rather elementary arguments to improve and generalise
the bounds~\eqref{eq:v 3/4}, \eqref{eq:v 7/12}
and~\eqref{eq:v 1/2 a.a.} and show that
in fact~\eqref{eq:v 1/2 a.a.}
holds for all $a$ with  $\gcd(a,m)=1$ and
also prove a stronger bound for
integers $m$ which are almost squarefree.
More precisely, we  obtain the following results.

\begin{thm}
\label{thm:Gen m}
For an arbitrary integer $m\ge 2$, uniformly over integers $a$ with $\gcd(a,m)=1$, we have
$$
v_a(m) \le m^{1/2+o(1)},
$$
as $m \to \infty$.
\end{thm}

For an integer $m$ we denote by  $m^*$ its kernel, that is,
the product of all prime divisors of $m$.

\begin{thm}
\label{thm:Typical m}
For an arbitrary integer $m\ge 2$, uniformly over integers $a$ with $\gcd(a,m)=1$, we have
$$
v_a(m) \le t m^{5/12 + o(1)} ,
$$
where $t = m/m^*$.
\end{thm}

In particular, for a squarefree $m$ we have $m^* = m$, thus we have:

\begin{cor}
\label{cor:SF m}
For an arbitrary squarefree integer $m\ge 2$, uniformly over integers $a$ with $\gcd(a,m)=1$, we have
$$
v_a(m) \le m^{5/12 + o(1)}.
$$
\end{cor}

Finally,
a simple counting argument shows that $m^* = m^{1 + o(1)}$ for almost
all $m$ and thus leads to the following estimate:

\begin{cor}
\label{cor:Almost all m}
For $M \to \infty$ and all but $o(M)$ positive  integers $m\le M$, uniformly over integers $a$ with $\gcd(a,m)=1$, we have
$$
v_a(m) \le m^{5/12 + o(1)}.
$$
\end{cor}

\section{Distribution of Points on Curves}

We denote
$$
\NU= \{(x,y) \ : \  xy \equiv a \pmod m,\ 1\le x\le U,\,1\le y \le
V\}.
$$

We need the following asymptotic formula on $\NU$ that is
immediate from the Weil bound of Kloosterman sums; see, for
example,~\cite{FuKi} (we note that in~\cite{FuKi} it is given only
for $a=1$ but the proof extends to arbitrary $a$ with
$\gcd(a,m)=1$ at  the cost of only obvious typographical
adjustments).

\begin{lem}
\label{lem:Asymp}  Uniformly over integers $a,U,V$,
$$
\NU
 =   UV
\frac{\varphi(m)}{m^2} +O\(m^{1/2 + o(1)}\).
$$
\end{lem}

We prove the following statement in a much more
general form that we need for our purpose as we believe
this can be of independent interest.

\begin{lem}
\label{lem:Curve} Let $\mu_i(X,Y) = X^{h_i}Y^{k_i}$, $i=1, \ldots,
s$,  be $s$ arbitrary distinct monomials. Assume that for a set of
$K\ge s$ distinct points $(x_\nu, y_\nu) \in \Z^2$ with
$\max\{|x_\nu|, |y_\nu|\} \le H$, $\nu =1, \ldots, K$, over an
arbitrary field $\F$ we have
$$
\det \(\mu_i(x_{\nu_j},y_{\nu_j})\)_{i,j=1}^s = 0
$$
for any $1 \le \nu_1 < \ldots < \nu_s \le K$.
Then there is a polynomial $F$ of the form
$$
F(X,Y) = \sum_{i=1}^s A_i\mu_i(X,Y)
$$
with integer coefficients satisfying $|A_i| \le H^{O(1)}$, $i=1,
\ldots, s$,  where the implied constant depends only on $s$,
and such that $F(x_\nu, y_\nu)=0$, $\nu =1, \ldots, K$.
\end{lem}

\begin{proof}
Let $r$ be the largest rank of all matrices $\(\mu_i(x_{\nu_j},y_{\nu_j}\)_{i,j=1}^s$
with $1 \le \nu_1 < \ldots < \nu_s \le K$. We have $1 \le r \le s-1$.
Without loss of generality we can assume that the matrix
$$
M = \det \(\mu_i(x_{j},y_{j})\)_{i,j=1}^{r+1,r}
$$ is of rank $r$.
Thus, there is a unique nontrivial vanishing linear combination
of columns with  relatively prime coefficients
$a_1, \ldots, a_{r+1}$ such that the first
non-zero coefficient is $1$. Furthermore, it is obvious (from the explicit
expression for solutions of system of linear equations via determinants
and trivial upper bounds on these  determinants), that
$|a_i| \le H^{O(1)}$, $i=1, \ldots, r+1$

Thus for any $\nu=1, \ldots, K$ the matrix obtained from $M$ by adding the bottom row
$(\mu_1(x_\nu,y_\nu), \ldots, \mu_r(x_\nu,y_\nu))$ is also of rank $k$,
so
$$a_1\mu_1(x_\nu,y_\nu)+ \ldots +a_{r+1} \mu_{r+1}(x_\nu,y_\nu)=0,
$$
which concludes the proof.
\end{proof}

\begin{lem}
\label{lem:Quadratic}
Let
$$
G(X,Y)= AX^2  + BXY + C Y^2 + DX + EY + F\in \Z[X,Y]
$$
be an irreducible quadratic polynomial with coefficients of
size at most $H$. Assume that $G(X,Y)$ is not affine
equivalent to a parabola $Y = X^2$ and has a nonzero determinant
$$
\Delta = B^2-4AC \ne 0.
$$
 Then the equation $G(x,y)=0$ has
at most $H^{o(1)}$ integral solutions $(x,y) \in [0,H]\times[0,H]$.
\end{lem}

\begin{proof}

%
The proof is based on the reduction of the equation $G(x,y)=0$
to a Pell equation $X^2 - UY^2 = V$ with some integers $U$ and $V$ of size $H^{O(1)}$
together with the estimate of  R.~C.~Vaughan and
T.~D.~Wooley~\cite[Lemma~3.5]{VaWo}
on  the number of solutions of this equation of a given size.

In the case when  the discriminant $\Delta$ is not a perfect square
the above reduction is given  by J.~Cilleruelo and
M.~Z.~Garaev~\cite[Proposition~1]{CillGar}. If $\Delta$ is a perfect square
it is obtained by V.~Shelestunova~\cite[Theorem~1]{Shel}.
\end{proof}

\section{Integral Polygons}

We say that a polygon $\cP \subseteq \R^2$ is integral if
all its vertices belong to the integral lattice $\Z^2$.

Also, following V.~I.~Arnold~\cite{Arn} we say two  polygons $\cP, \cQ \subseteq \R^2$ are equivalent
is there is an affine transformation
$$T: \vec{x} \mapsto A\vec{x} + \vec{b}, \qquad \vec{x} \in \R^2
$$
for $A = \GL_2(\Z)$ and $\vec{b} \in \Z^2$ preserving the integral
lattice $\Z^2$ (that is,  $\det A = \pm 1$) that maps $\cP$ to
$\cQ$.

We need the following result of I.~B{\'a}r{\'a}ny and
J.~Pach~\cite[Lemma~3]{BarPach}:

\begin{lem}
\label{lem:Equiv}  An integral polygon of area $S$ is equivalent to
a polygon contained in some box $[0,u]\times[0,v]$ of area $uv\le 4 S$.
\end{lem}

We note that it can also be derived (with a slightly weaker constant)
from a result of
V.~I.~Arnold~\cite[Lemma~1 of Section~2]{Arn} that asserts
that any integral convex polygon of area $S$ can be covered by an
integral  parallelogram of area at most $6S$.

We also recall the following general result of F.~V.~Petrov~\cite[Lemma 2.2]{Pet}
which we use only in $\R^2$. We use $\vol \fA$ to denote the volume
of a compact set $\fA\subseteq \R^d$

\begin{lem}
\label{lem:Pet}
Let $\fU\subseteq \R^d$ be a convex compact. We consider
a finite sequence of compacts $\fV_i\subseteq K$, $i=1, \ldots, n$,
such that none of them meets the convex hull of others. Then
$$\sum_{i=1}^n (\vol \fV_i)^{(d-1)/(d+1)}\ll (\vol \fU)^{(d-1)/(d+1)},$$
where the implied constant depends only on $d$.
\end{lem}

\section{Proof of Theorem~\ref{thm:Gen m}}

We estimate the number of  vertices $(x,y)$ of $\Cm$
that are inside of the square $[0, m/2]\times[0, m/2]$. The other three squares
\begin{equation}
\label{eq:Squares}
[0,m/2] \times [m/2,m],\quad
[m/2,m] \times [0,m/2], \quad
[m/2,m]\times [m/2,m]
\end{equation}
can be dealt with fully analogously.

We fix some $\varepsilon>0$ and also recall the well-known estimates
on the divisors and  Euler functions
\begin{equation}
\label{eq:phitau}
 \tau(s) = s^{o(1)}
 \mand  \varphi(s) =s^{1+o(1)},
\end{equation}
as $s \to \infty$,
see~\cite[Theorems~317 and~328]{HW},  we obtain our main technical result.

We claim that, for a sufficiently large $m$ we have
\begin{equation}
\label{eq:Cond}
xy \le m^{3/2 + \varepsilon}.
\end{equation}
for each such vertex.
Indeed, assume that condition~\eqref{eq:Cond} fails.

Then applying Lemma~\ref{lem:Asymp} to
$\Hm$ with $U = x m^{-\varepsilon/4}$ and $V = y m^{-\varepsilon/4}$,
we see that there are points $\vec{w}_j$, $j = 1, 2, 3, 4$,
in each of the translates of the box $[0,U] \times [0,V]$
to the corners of the $[0, m]\times[0, m]$ square.

Therefore the point $(x,y)$ is inside of the
convex hull of the points $\vec{w}_j$, $j = 1, 2, 3, 4$,
but is different from all of them,
and thus cannot be a point on $\Cm$.

We now see that there is
some integer $A$ with $1 \le A < m$ such that
for $(x,y)\in \Cm$ we have
$$
xy = A + m \ell
$$
with some nonnegative integer $\ell \le m^{3/2 -1 + \varepsilon}$.
When such an integer $k$ is fixed, by~\eqref{eq:phitau} there are $m^{o(1)}$
possibilities for the point $(x,y)$ and
the result now follows.

\section{Proof of Theorem~\ref{thm:Typical m}}

Fix some $\varepsilon > 0$.

As in the proof of Theorem~\ref{thm:Gen m}   we see
from Lemma~\ref{lem:Asymp} that all vertices $(u,v)$ on
$\Cm$ that are also inside of the square $[0,m/2] \times [0,m/2]$ satisfy
$$
uv \le  m^{3/2+ o(1)}.
$$
We estimate the number of such points.

The number of vertices  of $\Cm$ inside of the squares~\eqref{eq:Squares}
can be estimated fully analogously.

Hence,  it is
enough to estimate the number of vertices of $\Cm$ inside of
each of the boxes $[1,U]\times[1,V]$ with
$U = 2^j$, $V = m^{3/2+\varepsilon}2^{-j}$, $j  =1, 2, \ldots$.
Since only $O(\log m)$ such boxes are of our interest.

Let $\vec{v}_1,\ldots,\vec{v}_r\in \Cm$ be located in
$[1,U]\times[1,V]$. Assume that $r \ge t m^\varepsilon$ as otherwise
there is nothing to prove.
Select
$$
k = \fl{t m^{\varepsilon}}.
$$

By Lemma~\ref{lem:Pet}, there are $k$ consecutive vertices
$\vec{v}_{j+1},\ldots,\vec{v}_{j+k}$ such that the area of the
polygon formed by these vertices is bounded by
\begin{equation}
\label{eq:Q-bound}
Q= O(UV (r/k)^{-3}) = O(m^{3/2+4\varepsilon} t^3 r^{-3}).
\end{equation}
In particular, we have $k \ge 5$ for a sufficiently large $m$.

By Lemma~\ref{lem:Equiv}, we have an affine
transformation of $\R^2$ preserving $\Z^2$ such that the images of
all points $\vec{v}_{j+\nu}$ are points  $(X_\nu,Y_\nu) \in
[0,u]\times[0,v]$, $\nu =1, \ldots, k$ for some real positive $u$
and $v$ with $uv \ll Q$.

Note that all these  points  satisfy the congruence
\begin{equation}
\label{eq:Congr}
f(X_\nu,Y_\nu)\equiv 0 \pmod m, \qquad \nu =1, \ldots, k,
\end{equation}
where $f$ is a nonzero modulo $m$ quadratic polynomial
(which is the image of $XY-a$ under the above transformation).

Without loss of generality, we can assume that $X_k=Y_k=0$. So,
the constant term of $f$ is $0$. Take arbitrary
$\nu_1<\ldots<\nu_5$. The matrix
$$W = (X_{\nu_i}^2, X_{\nu_i} Y_{\nu_i}, Y_{\nu_i}^2,
X_{\nu_i}, Y_{\nu_i})_{i=1,\ldots,5}$$ is singular modulo $m$
since $f(X_{\nu_i},Y_{\nu_i})\equiv 0 \pmod m$, $i=1, \ldots, 5$.
This implies that  the determinant $\det W$ is divisible by $m$. Examining the
structure of the terms of $\det W$ one also sees that $\det W =
O(Q^4)$.

Therefore, if  $Q\le cm^{1/4}$ with an appropriate constant $c$
then $\det W = 0$ (over  $\Z$). We now see from
Lemma~\ref{lem:Curve} that there is a nonzero quadratic polynomial
$F(X,Y)$ such that
\begin{equation}
\label{eq:Equation}
F(X_\nu,Y_\nu)=0, \qquad \nu =1, \ldots, k,
\end{equation}
with the integer coefficients of size $m^{O(1)}$. Moreover, we may assume that the coefficients
of $F$ are relatively prime.

Let $\vec{v}_{j+\nu}=(x_\nu,y_\nu)$, $\nu=1,\ldots,k$. The
equation~\eqref{eq:Equation} is equivalent to the equation
\begin{equation}
\label{eq:Equation2}
G(x_\nu,y_\nu)=0, \qquad \nu =1, \ldots, k,
\end{equation}
for some quadratic polynomial $G(X,Y)\in \Z[X,Y]$ with relatively prime coefficients.
Next, we consider the
polynomial $H(X)=X^2G(X,a/X)$ over the ring of residues modulo
$m$. For any $\nu=1,\ldots,k$ we have $H(x_\nu)\equiv 0\pmod m$.

We take an arbitrary prime divisor $p>5$ of $m^*$. Assume that all coefficients
of $H$ are divisible by $p$. Then
any solution of the congruence $xy\equiv a\bmod p$ also satisfies the
congruence $G(x,y)\equiv 0\bmod p$. Therefore, there are at least $p-1>4$
common zeros of polynomials $xy-a$ and $G$ modulo $p$.  By
the B\'ezout Theorem, see, for example,~\cite[Section~5.3]{Fult},
the polynomial $G$ is a multiple of $xy-a$ modulo $p$.
Then $G$ is irreducible and
is not affine equivalent to a parabola modulo $p$. Consequently,
$G$ is irreducible and is not affine equivalent to a parabola over $\Z$
and also has a nonzero determinant.
Thus, we can apply Lemma~\ref{lem:Quadratic} and conclude that
the equation $G(x,y)=0$ has
at most $m^{o(1)}$ integral solutions $(x,y) \in [0,m]\times[0,m]$.
Now assume that for any prime divisor $p>5$ of $m^*$ there is a coefficient
of $H$ not divisible by $p$.  Using
the Chinese Remainder Theorem, we see that
the congruence
$H(x)\equiv 0\pmod m$, $1 \le x \le m$, has at most
$t4^{\omega(m^*)}  = t \tau(m^*)^2$ solutions, where $\omega(s)$
is the number of prime divisors of an integer  $s$.
Recalling~\eqref{eq:phitau} we see that  in both cases
$k = t m^{o(1)}$ which contradicts to our choice of  $k =
\fl{t m^{\varepsilon}}$. Therefore $Q>cm^{1/4}$ which together
with~\eqref{eq:Q-bound} implies $r = O\(t m^{5/12+4\varepsilon/3}\)$.
Since $\varepsilon > 0$ is arbitrary, the result now follows.

\section{Comments}

It is shown in~\cite{Shp3} that for almost all residue classes
$a$ modulo $m$, the asymptotic formula of Lemma~\ref{lem:Asymp}
can be improved. Perhaps this can be used to improve the
bound of  Theorems~\ref{thm:Gen m}
and~\ref{thm:Typical m} on average over $a$.

Convex hull of the points on multidimensional hyperbolas
can be studied as well. In fact in the multidimensional case
a different technique can be used to obtain versions of
Lemma~\ref{lem:Asymp} which have no analogues in the two dimensional case, see~\cite{Shp2}. Furthermore, the method of proof of
Theorem~\ref{thm:Gen m}  easily extends to the multidimensional
case as well. However extending the method of proof
of  Theorem~\ref{thm:Typical m} seems to be more difficult and
we pose this as an open question.

\section*{Acknowledgement}

The authors would like to thank Imre B{\'a}r{\'a}ny
for the information about the existence of
a proof of Lemma~\ref{lem:Equiv} in~\cite{BarPach}
and Oleg German for an alternative proof of this result.
The authors are also grateful  to
P{\"a}r Kurlberg for supplying then with an alternative proof
of Lemma~\ref{lem:Quadratic} and
to David McKinnon and Alfred Menezes for
useful discussion and for the information about the thesis of
V.~Shelestunova~\cite{Shel}.

The research of S.~K. was supported in part
Russian Fund of Basic Researches  Grant N.~11-01-00329
from the Russian Fund of Basic Researches
 and that of I.~S. by ARC grant  DP1092835.

\end{document}